\documentclass[12pt,twoside]{amsart}
\usepackage{amsmath, amsthm, amscd, amsfonts, amssymb, graphicx}
\usepackage{enumerate}
\usepackage[colorlinks=true,
linkcolor=magenta,
urlcolor=cyan,
citecolor=cyan]{hyperref}

\newtheorem{theorem}{Theorem}

\newtheorem{lemma}{Lemma}[section]
\newtheorem{remark}{Remark}[section]

\newtheorem{corollary}{Corollary}[section]
\newtheorem{example}{Example}[section]

\numberwithin{equation}{section}

\begin{document}
\title{Sharpening Some classical numerical radius inequalities}
\author{Hamid Reza Moradi$^1$, Mohsen Erfanian Omidvar$^2$ and Khalid Shebrawi$^3$}

\maketitle
\begin{abstract}
New upper and lower bounds for the numerical radii of Hilbert space
operators are given. Among our results, we prove that if  $A\in \mathcal{B}%
\left( \mathcal{H}\right) $ is a hyponormal operator, then for all non-negative non-decreasing operator convex $f$ on $ [0,\infty ),$ we have 
\[f\left( \omega \left( A \right) \right)\le \frac{1}{2}\left\| f\left( \frac{1}{1+\frac{\xi _{\left| A \right|}^{2}}{8}}\left| A \right| \right)+f\left( \frac{1}{1+\frac{\xi _{\left| A \right|}^{2}}{8}}\left| {{A}^{*}} \right| \right) \right\|,\]
where ${{\xi }_{\left\vert A\right\vert }}=\underset{\left\Vert x\right\Vert
=1}{\mathop{\inf }}\,\left\{ \frac{\left\langle \left( \left\vert
A\right\vert -\left\vert {{A}^{\ast }}\right\vert \right) x,x\right\rangle }{%
\left\langle \left( \left\vert A\right\vert +\left\vert {{A}^{\ast }}%
\right\vert \right) x,x\right\rangle }\right\} $. Our results refine and
generalize earlier inequalities for hyponormal operator.
\end{abstract}

\pagestyle{myheadings} 
\markboth{\centerline {Sharpening some classical numerical radius inequalities}}
{\centerline {H.R. Moradi, M.E. Omidvar \& K. Shebrawi}} \bigskip \bigskip 

\section{\bf Introduction}
\vskip0.4 true cm
Let $\left( \mathcal{H},\left\langle \cdot ,\cdot
\right\rangle \right) $ be a complex Hilbert space and $\mathcal{B}\left( 
\mathcal{H}\right) $ denote the ${{C}^{\ast }}$-algebra of all bounded
linear operators on $\mathcal{H}$. For $A\in \mathcal{B}\left( \mathcal{H} \right)$, we denote by $\left| A \right|$ the absolute value operator of $A$, that is, $\left| A \right|={{\left( {{A}^{*}}A \right)}^{\frac{1}{2}}}$, where ${{A}^{*}}$ is the adjoint operator of $A$.
A continuous real-valued function $f$ defined on an interval $I$ is said to be operator convex if $f\left( \lambda A+\left( 1-\lambda  \right)B \right)\le \lambda f\left( A \right)+\left( 1-\lambda  \right)f\left( B \right)$ for all self-adjoint operators $A,B$ with spectra contained in $I$ and all $\lambda \in \left[ 0,1 \right]$.

The numerical range of an operator $A$ in $\mathcal{B}\left( \mathcal{H} \right)$ is defined as $W\left( A \right)=\left\{ \left\langle Ax,x \right\rangle :\text{ }\left\| x \right\|=1 \right\}$. For any $A\in \mathcal{B}\left( \mathcal{H} \right)$, $\overline{W\left( A \right)}$ is a convex subset of the complex plane containing the spectrum of $A$ (see {{\cite[Chapter 2]{book2}}}). 

Recall that $\omega \left( A\right) =%
\underset{\left\Vert x\right\Vert =1}{\mathop{\sup }}\,\left\vert
\left\langle Ax,x\right\rangle \right\vert $ and $\left\Vert A\right\Vert =%
\underset{\left\Vert x\right\Vert =1}{\mathop{\sup }}\,\left\Vert
Ax\right\Vert $. It is well-known that $\omega \left( \cdot \right) $
defines a norm on $\mathcal{B}\left( \mathcal{H}\right) $, which is
equivalent to the usual operator norm $\left\Vert \cdot \right\Vert $.
Namely, for $A\in \mathcal{B}\left( \mathcal{H}\right) $, we have 
\begin{equation}
\frac{1}{2}\left\Vert A\right\Vert \leq \omega \left( A\right) \leq
\left\Vert A\right\Vert .  \label{7}
\end{equation}%
Other facts about the numerical radius that we use can be found in \cite{5}.

The inequalities in \eqref{7} have been improved considerably by many
authors, (see, e.g., \cite{1, HKS1, HKS2,SA,S,Y}), Kittaneh \cite{3,2} has
shown the following precise estimates of $\omega \left( A\right) $ by using
several norm inequalities and ingenious techniques: 
\begin{equation}
\omega \left( A\right) \leq \frac{1}{2}\left( \left\Vert A\right\Vert +{{%
\left\Vert {{A}^{2}}\right\Vert }^{\frac{1}{2}}}\right) ,  \label{5}
\end{equation}%
and 
\begin{equation}
\frac{1}{4}\left\Vert {{\left\vert A\right\vert }^{2}}+{{\left\vert {{A}%
^{\ast }}\right\vert }^{2}}\right\Vert \leq {{\omega }^{2}}\left( A\right)
\leq \frac{1}{2}\left\Vert {{\left\vert A\right\vert }^{2}}+{{\left\vert {{A}%
^{\ast }}\right\vert }^{2}}\right\Vert .  \label{3}
\end{equation}%
In \cite{4}, Dragomir gave the following estimate of the numerical radius
which refines the second inequality in \eqref{7}: For every $A$, 
\begin{equation*}
{{\omega }^{2}}\left( A\right) \leq \frac{1}{2}\left( \omega \left( {{A}^{2}}%
\right) +{{\left\Vert A\right\Vert }^{2}}\right) .
\end{equation*}%
In this paper, we establish a considerable improvement of the second
inequality in \eqref{3}. We also propose a new upper bound for $\omega
\left( \cdot \right) $ for the hyponormal operators. Next, we will give a
refinement of the first inequality in \eqref{7}.

\section{\bf Upper bounds for the numerical radii}

\vskip 0.4 true cm The following lemma is known as the mixed Schwarz
inequality (see {{\cite[pp. 75-76]{12}}}).

\begin{lemma}
\label{13} If $A\in \mathcal{B}\left( \mathcal{H}\right) $, then 
\begin{equation*}
\left\vert \left\langle Ax,y\right\rangle \right\vert \leq {{\left\langle
\left\vert A\right\vert x,x\right\rangle }^{\frac{1}{2}}}{{\left\langle
\left\vert A^{\ast }\right\vert y,y\right\rangle }^{\frac{1}{2}}},
\end{equation*}%
for all $x,y\in \mathcal{H}$.
\end{lemma}

The second lemma is a norm inequality for the sum of two positive operators,
which can be found in \cite{NORM}.

\begin{lemma}
\label{l2}If $A$ and $B$ are positive operators in $\mathcal{B}\left( \mathcal{H}%
\right) $, then%
\begin{equation*}
\left\Vert A+B\right\Vert \leq \max \left( \left\Vert A\right\Vert
,\left\Vert B\right\Vert \right) +\left\Vert A^{\frac{1}{2}}B^{\frac{1}{2}%
}\right\Vert .
\end{equation*}
\end{lemma}
 The following lemma contains a simple inequality, which will be needed in the sequel.
\begin{lemma}
For each $\alpha \ge 1$, we have
\begin{equation}  \label{41}
\frac{\alpha -1}{\alpha +1}\le \ln \alpha .
\end{equation}
\end{lemma}
\begin{proof}
Taking $f\left( \alpha  \right)\equiv \ln \alpha -\frac{\alpha -1}{\alpha +1}$, where $\alpha \ge 1$. By an elementary computation we have $f'\left( \alpha  \right)\ge 0$, so $f\left( \alpha  \right)$ is an increasing function for $\alpha \ge 1$. On the other hand $f\left( \alpha  \right)\ge f\left( 1 \right)=0$.
\end{proof}

\vskip 0.3 true cm
 Now, we are ready to present our new improvement of the
second inequality in \eqref{3}. Recall that, an operator $A$ defined on a
Hilbert space $\mathcal{H}$ is said to be hyponormal if ${{A}^{*}}A-A{{A}^{*}%
}\ge 0$, or equivalently if $\left\| {{A}^{*}}x \right\|\le \left\| Ax
\right\|$ for every $x\in \mathcal{H}$.

\begin{theorem}
\label{44} Let $A\in \mathcal{B}\left( \mathcal{H}\right) $ be a hyponormal
operator. Then, for all non-negative non-decreasing operator convex $f$ on $[0,\infty ),$ we have 
\begin{equation}
f\left( \omega \left( A\right) \right) \leq \frac{1}{2}\left\Vert f\left( 
\frac{1}{1+\frac{\xi _{\left\vert A\right\vert }^{2}}{8}}\left\vert
A\right\vert \right) +f\left( \frac{1}{1+\frac{\xi _{\left\vert A\right\vert
}^{2}}{8}}\left\vert A^{\ast }\right\vert \right) \right\Vert ,  \label{43}
\end{equation}%
where ${{\xi }_{\left\vert A\right\vert }}=\underset{\left\Vert x\right\Vert
=1}{\mathop{\inf }}\,\left\{ \frac{\left\langle \left( \left\vert
A\right\vert -\left\vert {{A}^{\ast }}\right\vert \right) x,x\right\rangle }{%
\left\langle \left( \left\vert A\right\vert +\left\vert {{A}^{\ast }}%
\right\vert \right) x,x\right\rangle }\right\} $.
\end{theorem}

\begin{proof}
Since $A$ is a hyponormal operator we have $1\le \frac{\left\langle \left| A
\right|x,x \right\rangle }{\left\langle \left| {{A}^{*}} \right|x,x
\right\rangle }$, for each $x\in \mathcal{H}$. On choosing $\alpha =\frac{%
\left\langle \left| A \right|x,x \right\rangle }{\left\langle \left| {{A}^{*}%
} \right|x,x \right\rangle }$ in \eqref{41} we get 
\begin{equation*}
\left( 0\le \right)\text{ }\frac{\left\langle \left( \left| A \right|-\left| 
{{A}^{*}} \right| \right)x,x \right\rangle }{\left\langle \left( \left| A
\right|+\left| {{A}^{*}} \right| \right)x,x \right\rangle }\le \ln \frac{%
\left\langle \left| A \right|x,x \right\rangle }{\left\langle \left| {{A}^{*}%
} \right|x,x \right\rangle }.
\end{equation*}
Whence 
\begin{equation}  \label{42}
\underset{\left\| x \right\|=1}{\mathop{\inf }}\,\frac{\left\langle \left(
\left| A \right|-\left| {{A}^{*}} \right| \right)x,x \right\rangle }{%
\left\langle \left( \left| A \right|+\left| {{A}^{*}} \right| \right)x,x
\right\rangle }\le \ln \frac{\left\langle \left| A \right|x,x \right\rangle 
}{\left\langle \left| {{A}^{*}} \right|x,x \right\rangle }.
\end{equation}
We denote the expression on the left-hand side of \eqref{42} by ${{\xi }%
_{\left| A \right|}}$. On the other hand Zou et al. in \cite{18} proved that for each $a,b>0$, 
\begin{equation*}
\left( 1+\frac{{{\left( \ln a-\ln b\right) }^{2}}}{8}\right) \sqrt{ab}\leq 
\frac{a+b}{2}.
\end{equation*}%
By taking $a=\left\langle \left\vert A\right\vert x,x\right\rangle $ and $%
b=\left\langle \left\vert {{A}^{\ast }}\right\vert x,x\right\rangle $ and
taking into account that ${{\xi }_{\left\vert A\right\vert }}\leq \ln \frac{%
\left\langle \left\vert A\right\vert x,x\right\rangle }{\left\langle
\left\vert {{A}^{\ast }}\right\vert x,x\right\rangle }$, we infer that 
\begin{equation*}
\sqrt{\left\langle \left\vert A\right\vert x,x\right\rangle \left\langle
\left\vert {{A}^{\ast }}\right\vert x,x\right\rangle }\leq \frac{1}{2\left(
1+\frac{\xi _{\left\vert A\right\vert }^{2}}{8}\right) }\left\langle \left(
\left\vert A\right\vert +\left\vert {{A}^{\ast }}\right\vert \right)
x,x\right\rangle .
\end{equation*}%
By using Lemma \ref{13}, we get 
\begin{equation*}
\left\vert \left\langle Ax,x\right\rangle \right\vert \leq \frac{1}{2\left(
1+\frac{\xi _{\left\vert A\right\vert }^{2}}{8}\right) }\left\langle \left(
\left\vert A\right\vert +\left\vert {{A}^{\ast }}\right\vert \right)
x,x\right\rangle .
\end{equation*}%
Now, by taking supremum over $x\in \mathcal{H},\left\Vert x\right\Vert =1,$
we get%
\begin{equation*}
\omega \left( A\right) \leq \frac{1}{2\left( 1+\frac{\xi _{\left\vert
A\right\vert }^{2}}{8}\right) }\left\Vert \left\vert A\right\vert
+\left\vert {{A}^{\ast }}\right\vert \right\Vert .
\end{equation*}%
Therefore,
\[\begin{aligned}
   f\left( \omega \left( A \right) \right)&\le f\left( \frac{1}{2\left( 1+\frac{\xi _{\left| A \right|}^{2}}{8} \right)}\left\| \left| A \right|+\left| {{A}^{*}} \right| \right\| \right) \\ 
 & =\left\| f\left( \frac{1}{2\left( 1+\frac{\xi _{\left| A \right|}^{2}}{8} \right)}\left| A \right|+\frac{1}{2\left( 1+\frac{\xi _{\left| A \right|}^{2}}{8} \right)}\left| {{A}^{*}} \right| \right) \right\| \\ 
 & \le \frac{1}{2}\left\| f\left( \frac{1}{1+\frac{\xi _{\left| A \right|}^{2}}{8}}\left| A \right| \right)+f\left( \frac{1}{1+\frac{\xi _{\left| A \right|}^{2}}{8}}\left| {{A}^{*}} \right| \right) \right\|.  
\end{aligned}\]
This completes the proof.
\end{proof}
\begin{remark}
Notice that, if $A$ is a normal operator, then ${{\xi }_{\left\vert
A\right\vert }}=0$.
\end{remark}

An important special case of Theorem \ref{44}, which leads to an improvement
and a generalization of inequality \eqref{3} for hyponormal operators, can be
stated as follows.

\begin{corollary}\label{2.1}
Let $A\in \mathcal{B}\left( \mathcal{H}\right) $ be a hyponormal operator.
Then, for all $1\le r\le 2$ we have 
\begin{equation*}
\omega ^{r}\left( A\right) \leq \frac{1}{2\left( 1+\frac{\xi _{\left\vert
A\right\vert }^{2}}{8}\right) ^{r}}\left\Vert \left\vert A\right\vert
^{r}+\left\vert A^{\ast }\right\vert ^{r}\right\Vert ,
\end{equation*}%
where ${{\xi }_{\left\vert A\right\vert }}=\underset{\left\Vert x\right\Vert
=1}{\mathop{\inf }}\,\left\{ \frac{\left\langle \left( \left\vert
A\right\vert -\left\vert {{A}^{\ast }}\right\vert \right) x,x\right\rangle }{%
\left\langle \left( \left\vert A\right\vert +\left\vert {{A}^{\ast }}%
\right\vert \right) x,x\right\rangle }\right\} $. In particular,%
\begin{equation}
\omega \left( A\right) \leq \frac{1}{2\left( 1+\frac{\xi _{\left\vert
A\right\vert }^{2}}{8}\right) }\left\Vert \left\vert A\right\vert
+\left\vert A^{\ast }\right\vert \right\Vert ,  \label{1}
\end{equation}%
and%
\begin{equation*}
\omega ^{2}\left( A\right) \leq \frac{1}{2\left( 1+\frac{\xi _{\left\vert
A\right\vert }^{2}}{8}\right) ^{2}}\left\Vert A^{\ast }A+AA^{\ast
}\right\Vert .
\end{equation*}%
An operator norm inequality which will be used in next corollary says that
for any positive operators $A$,$B\in \mathcal{B}\left( \mathcal{H}\right) ,$ we have
(see \cite{BH}) 
\begin{equation}
\left\Vert A^{r}B^{r}\right\Vert \leq \left\Vert AB\right\Vert ^{r},\qquad \text{ \
for all }  0\leq r\leq 1.  \label{9}
\end{equation}%
The following result refines and generalizes inequality \eqref{5} for
hyponormal operators.
\end{corollary}

\begin{corollary}
\label{c1}Let $A\in \mathcal{B}\left( \mathcal{H}\right) $ be a hyponormal
operator. Then 
\begin{equation*}
\omega ^{r}\left( A\right) \leq \frac{1}{2\left( 1+\frac{\xi _{\left\vert
A\right\vert }^{2}}{8}\right) ^{r}}\left( \left\Vert A\right\Vert
^{r}+\left\Vert \left\vert A\right\vert ^{\frac{r}{2}}\left\vert A^{\ast
}\right\vert ^{\frac{r}{2}}\right\Vert \right),
\end{equation*}%
for all $1\le r\le 2$. In particular
\begin{equation*}
\omega ^{r}\left( A\right) \leq \frac{1}{2\left( 1+\frac{\xi _{\left\vert
A\right\vert }^{2}}{8}\right) ^{r}}\left( \left\Vert A\right\Vert
^{r}+\left\Vert A^{2}\right\Vert ^{\frac{r}{2}}\right),
\end{equation*}%
for $1\leq r\leq 2.$
\end{corollary}

\begin{proof}
Applying Corollary \ref{2.1} and Lemma \ref{l2}, we have
\[\begin{aligned}
   {{\omega }^{r}}\left( A \right)&\le \frac{1}{2{{\left( 1+\frac{\xi _{\left| A \right|}^{2}}{8} \right)}^{r}}}\left\| {{\left| A \right|}^{r}}+{{\left| {{A}^{*}} \right|}^{r}} \right\| \\ 
 & \le \frac{1}{2{{\left( 1+\frac{\xi _{\left| A \right|}^{2}}{8} \right)}^{r}}}\left( \max \left( {{\left\| A \right\|}^{r}},{{\left\| {{A}^{*}} \right\|}^{r}} \right)+\left\| {{\left| A \right|}^{\frac{r}{2}}}{{\left| {{A}^{*}} \right|}^{\frac{r}{2}}} \right\| \right) \\ 
 & =\frac{1}{2{{\left( 1+\frac{\xi _{\left| A \right|}^{2}}{8} \right)}^{r}}}\left( {{\left\| A \right\|}^{r}}+\left\| {{\left| A \right|}^{\frac{r}{2}}}{{\left| {{A}^{*}} \right|}^{\frac{r}{2}}} \right\| \right).  
\end{aligned}\]
For the particular applying inequality \eqref{9}, we have 
\begin{equation*}
\left\Vert \left\vert A\right\vert ^{\frac{r}{2}}\left\vert A^{\ast
}\right\vert ^{\frac{r}{2}}\right\Vert \leq \left\Vert \left\vert
A\right\vert \left\vert A^{\ast }\right\vert \right\Vert ^{\frac{r}{2}%
}=\left\Vert A^{2}\right\Vert ^{\frac{r}{2}},
\end{equation*}%
for $1\leq r\leq 2.$
\end{proof}

\vskip 0.3 true cm Recently, Kian \cite{8} improved Jensen's operator
inequality via superquadratic functions. As an application, he showed that
the following inequality is valid:

\begin{lemma}
\label{14} {{\cite[Example 3.6]{8}}} Let ${{A}_{1}},\ldots ,{{A}_{n}}$ be
positive operators, then 
\begin{equation*}
{{\left\| \sum\limits_{i=1}^{n}{{{w}_{i}}{{A}_{i}}} \right\|}^{r}}\le
\left\| \sum\limits_{i=1}^{n}{{{w}_{i}}A_{i}^{r}} \right\|-\underset{\left\|
x \right\|=1}{\mathop{\inf }}\,\left\{ \sum\limits_{i=1}^{n}{{{w}_{i}}%
\left\langle {{\left| {{A}_{i}}-\sum\limits_{j=1}^{n}{{{w}_{j}}\left\langle {%
{A}_{j}}x,x \right\rangle } \right|}^{r}}x,x \right\rangle } \right\},\qquad
r\ge 2,
\end{equation*}
for each ${{w}_{1}},\ldots ,{{w}_{n}}$ with $\sum\nolimits_{i=1}^{n}{{{w}_{i}%
}}=1$.
\end{lemma}

\vskip 0.3 true cm This, in turn, leads to the following:

\begin{theorem}
\label{19} Let $A\in \mathcal{B}\left( \mathcal{H} \right)$, then 
\begin{equation}  \label{2}
{{\omega }^{2}}\left( A \right)\le \frac{1}{2}\left( \left\| {{\left| A
\right|}^{2}}+{{\left| {{A}^{*}} \right|}^{2}} \right\|-\underset{\left\| x
\right\|=1}{\mathop{\inf }}\,\xi \left( x \right) \right),
\end{equation}
where $\xi \left( x \right)=\left\langle \left( {{\left| \left| A \right|-%
\frac{1}{2}\left\langle \left( \left| A \right|+\left| {{A}^{*}} \right|
\right)x,x \right\rangle \right|}^{2}}+{{\left| \left| {{A}^{*}} \right|-%
\frac{1}{2}\left\langle \left( \left| A \right|+\left| {{A}^{*}} \right|
\right)x,x \right\rangle \right|}^{2}} \right)x,x \right\rangle $.
\end{theorem}

\begin{proof}
One can easily see that for each $A\in \mathcal{B}\left( \mathcal{H}\right) $ we have
\begin{equation*}
\omega \left( A\right) \leq \frac{1}{2}\left\Vert \left\vert A\right\vert
+\left\vert {{A}^{\ast }}\right\vert \right\Vert ,
\end{equation*}%
we can also write 
\begin{equation}
{{\omega }^{2}}\left( A\right) \leq \frac{1}{4}{{\left\Vert \left\vert
A\right\vert +\left\vert {{A}^{\ast }}\right\vert \right\Vert }^{2}}.
\label{4}
\end{equation}%
Choosing $n,r=2,\text{ }{{w}_{1}}={{w}_{2}}=\frac{1}{2},\text{ }{{A}_{1}}%
=\left\vert A\right\vert $ and ${{A}_{2}}=\left\vert {{A}^{\ast }}%
\right\vert $ in Lemma \ref{14}, we infer 
\begin{equation*}
\begin{aligned} {{\left\| \left| A \right|+\left| {{A}^{*}} \right|
\right\|}^{2}}&\le 2 \left( \left\| {{\left| A \right|}^{2}}+{{\left|
{{A}^{*}} \right|}^{2}} \right\| \right.-\underset{\left\| x
\right\|=1}{\mathop{\inf }}\,\left\{ \left\langle {{\left| \left| A
\right|-\frac{1}{2}\left( \left\langle \left| A \right|x,x \right\rangle
+\left\langle \left| {{A}^{*}} \right|x,x \right\rangle \right)
\right|}^{2}}x,x \right\rangle \right. \\ &\quad \left. \left. +\left\langle
{{\left| \left| {{A}^{*}} \right|-\frac{1}{2}\left( \left\langle \left| A
\right|x,x \right\rangle +\left\langle \left| {{A}^{*}} \right|x,x
\right\rangle \right) \right|}^{2}}x,x \right\rangle \right\} \right). \\
\end{aligned}
\end{equation*}

It now follows from \eqref{4} that 
\begin{equation*}
\begin{aligned} {{\omega }^{2}}\left( A \right)&\le \frac{1}{4}{{\left\|
\left| A \right|+\left| {{A}^{*}} \right| \right\|}^{2}} \\ & \le
\frac{1}{2}\left( \left\| {{\left| A \right|}^{2}}+{{\left| {{A}^{*}}
\right|}^{2}} \right\| \right.-\underset{\left\| x \right\|=1}{\mathop{\inf
}}\,\left\{ \left\langle {{\left| \left| A \right|-\frac{1}{2}\left(
\left\langle \left| A \right|x,x \right\rangle +\left\langle \left|
{{A}^{*}} \right|x,x \right\rangle \right) \right|}^{2}}x,x \right\rangle
\right. \\ &\quad \left. \left. +\left\langle {{\left| \left| {{A}^{*}}
\right|-\frac{1}{2}\left( \left\langle \left| A \right|x,x \right\rangle
+\left\langle \left| {{A}^{*}} \right|x,x \right\rangle \right)
\right|}^{2}}x,x \right\rangle \right\} \right). \end{aligned}
\end{equation*}

The validity of this inequality is just Theorem \ref{19}.
\end{proof}

\begin{remark}
Notice that 
\begin{equation*}
\underset{\left\| x \right\|=1}{\mathop{\inf }}\,\xi \left( x \right)>0\text{
}\Leftrightarrow \text{ }0\notin \overline{W\left( {{\left| \left| A \right|-%
\frac{1}{2}\left\langle \left( \left| A \right|+\left| {{A}^{*}} \right|
\right)x,x \right\rangle \right|}^{2}}+{{\left| \left| {{A}^{*}} \right|-%
\frac{1}{2}\left\langle \left( \left| A \right|+\left| {{A}^{*}} \right|
\right)x,x \right\rangle \right|}^{2}} \right)}.
\end{equation*}
\end{remark}
\vskip 0.3 true cm
To make things a bit clearer, we consider the following example:
\begin{example}
Taking $A=\left( \begin{matrix}
   0 & 0  \\
   3 & 0  \\
\end{matrix} \right)$. By an easy computation we find that
	\[{{\left| \left| A \right|-\frac{1}{2}\left\langle \left( \left| A \right|+\left| {{A}^{*}} \right| \right)x,x \right\rangle  \right|}^{2}}+{{\left| \left| {{A}^{*}} \right|-\frac{1}{2}\left\langle \left( \left| A \right|+\left| {{A}^{*}} \right| \right)x,x \right\rangle  \right|}^{2}}=\left( \begin{matrix}
   4.5 & 0  \\
   0 & 4.5  \\
\end{matrix} \right).\] 
It is well-known that, $A=\lambda I$ if and only if $W\left( A \right)=\left\{ \lambda  \right\}$ (see, e.g., {{\cite[Section 18]{book}}}). So we get $\underset{\left\| x \right\|=1}{\mathop{\inf }}\,\xi \left( x \right)=4.5>0$.
\end{example}
This shows that the inequality \eqref{2} provides an improvement for the second inequality in \eqref{3}.
\section{\bf Lower bounds for the numerical radii}

\vskip 0.4 true cm The next theorem is slightly more intricate.

\begin{theorem}
\label{21} Let $A\in \mathcal{B}\left( \mathcal{H} \right)$, then 
\begin{equation}  \label{8}
\left\| A \right\|\left( 1-\frac{1}{2}{{\left\| I-\frac{A}{\left\| A \right\|} \right\|}^{2}} \right)\le \omega \left( A \right).
\end{equation}
\end{theorem}

\begin{proof}
It is easy to check that 
\begin{equation}  \label{15}
1-\frac{1}{2}{{\left\| \frac{x}{\left\| x \right\|}-\frac{y}{\left\| y
\right\|} \right\|}^{2}}\le \frac{1}{\left\| x \right\|\left\| y \right\|}%
\left| \left\langle x,y \right\rangle \right|,
\end{equation}
for every $x,y\in \mathcal{H}$.

If we choose $\left\| x \right\|=\left\| y \right\|=1$ in \eqref{15} we get 
\begin{equation}  \label{16}
1-\frac{1}{2}{{\left\| x-y \right\|}^{2}}\le \left| \left\langle x,y \right\rangle  \right|.
\end{equation}
This is an interesting inequality in itself as well. Now taking $y=\frac{Ax}{%
\left\| Ax \right\|}$ in \eqref{16}, we infer 
\begin{equation}  \label{17}
\left\| Ax \right\|\left( 1-\frac{1}{2}{{\left\| x-\frac{Ax}{\left\| Ax \right\|} \right\|}^{2}} \right)\le \left| \left\langle Ax,x \right\rangle  \right|.
\end{equation}
Since $\left\| x \right\|=1$, $\left\| Ax \right\|$ does not exceed $\left\| A \right\|$. Hence we get from \eqref{17} that 
\begin{equation*}
\left\| Ax \right\|\left( 1-\frac{1}{2}{{\left\| I-\frac{A}{\left\| A \right\|} \right\|}^{2}} \right)\le \left| \left\langle Ax,x \right\rangle  \right|.
\end{equation*}
Now by taking supremum over $x\in \mathcal{H}$ with $\left\| x \right\|=1$,
we deduce the desired inequality \eqref{8}.
\end{proof}

\begin{remark}
It is striking that if $\left\| A-\left\| A \right\| \right\|\le \left\| A
\right\|$, then inequality \eqref{8} provides an improvement for the first
inequality in \eqref{7}.
\end{remark}

\begin{example}
Taking $A=\left( \begin{matrix}
   2 & 1  \\
   0 & 4  \\
\end{matrix} \right)$. Then  $\left\| A \right\|\simeq 4.1594$ and $\left\| A-\left\| A \right\| \right\|\simeq 2.3807$. We obtain by easy computation 
\[\frac{1}{2}\left\| A \right\|\simeq 2.079,\qquad \left\| A \right\|\left( 1-\frac{1}{2}\left\| I-\frac{A}{\left\| A \right\|} \right\| \right)\simeq 2.968,\qquad \omega \left( A \right)\simeq 4.118,\]
whence
\[\frac{1}{2}\left\| A \right\|\lvertneqq \left\| A \right\|\left( 1-\frac{1}{2}\left\| I-\frac{A}{\left\| A \right\|} \right\| \right)\lvertneqq \omega \left( A \right),\]
which shows that if $\left\| A-\left\| A \right\| \right\|\le \left\| A
\right\|$, then inequality \eqref{8} is really an improvement of the first
inequality in \eqref{7}.
\end{example}
\vskip 0.3 true cm
 The following basic lemma is essentially known as in {{
\cite[Lemma 1]{13}}}, but our expression is a little bit different from
those in \cite{13}. For the sake of convenience, we give it a slim proof.
\begin{lemma}
Let $x,y,{{z}_{i}},\text{} i=1,\ldots ,n$ be nonzero vectors and $
\left\langle {{z}_{j}},{{z}_{i}} \right\rangle \ne 0$, then 
\begin{equation}  \label{30}
{{\left| \left\langle x-\sum\limits_{i}{\frac{\left\langle x,{{z}_{i}}
\right\rangle }{\sum\nolimits_{j}{\left| \left\langle {{z}_{j}},{{z}_{i}}
\right\rangle \right|}}{{z}_{i}}},y \right\rangle \right|}^{2}}\le {{\left\|
y \right\|}^{2}}\left( {{\left\| x \right\|}^{2}}-\sum\limits_{i}{\frac{{{%
\left| \left\langle x,{{z}_{i}} \right\rangle \right|}^{2}}}{%
\sum\nolimits_{j}{\left| \left\langle {{z}_{i}},{{z}_{j}} \right\rangle
\right|}}} \right).
\end{equation}
\end{lemma}

\begin{proof}
Define 
\begin{equation*}
u=x-\sum\limits_{i}{\frac{\left\langle x,{{z}_{i}} \right\rangle }{%
\sum\nolimits_{j}{\left| \left\langle {{z}_{j}},{{z}_{i}} \right\rangle
\right|}}{{z}_{i}}}.
\end{equation*}
Whence 
\begin{equation}  \label{31}
{{\left\| u \right\|}^{2}}={{\left\| x-\sum\limits_{i}{{{a}_{i}}{{z}_{i}}}
\right\|}^{2}}\le {{\left\| x \right\|}^{2}}-\sum\limits_{i}{\frac{{{\left|
\left\langle x,{{z}_{i}} \right\rangle \right|}^{2}}}{\sum\nolimits_{j}{%
\left| \left\langle {{z}_{i}},{{z}_{j}} \right\rangle \right|}}}.
\end{equation}
By multiplying both sides \eqref{31} by ${{\left\| y \right\|}^{2}}$ and
then utilizing the Cauchy Schwarz inequality we get 
\begin{equation*}
{{\left| \left\langle u,y \right\rangle \right|}^{2}}\le {{\left\| y \right\|%
}^{2}}\left( {{\left\| x \right\|}^{2}}-\sum\limits_{i}{\frac{{{\left|
\left\langle x,{{z}_{i}} \right\rangle \right|}^{2}}}{\sum\nolimits_{j}{%
\left| \left\langle {{z}_{i}},{{z}_{j}} \right\rangle \right|}}} \right),
\end{equation*}
which is exactly desired inequality \eqref{30}.
\end{proof}
\vskip 0.3 true cm Finally, we state the last result.
\begin{theorem}\label{d}
\label{32} Let $A\in \mathcal{B}\left( \mathcal{H} \right)$ be an invertible operator, then 
\begin{equation*}
\underset{\left\| x \right\|=1}{\mathop{\inf }}\,{{\xi }^{2}}\left( x
\right)+{{\omega }^{2}}\left( A \right)\le {{\left\| A \right\|}^{2}},
\end{equation*}
where $\xi \left( x \right)=\frac{\left| \left\langle {{A}^{2}}x,x
\right\rangle -{{\left\langle Ax,x \right\rangle }^{2}} \right|}{\left\| {{A}%
^{*}}x \right\|}$.
\end{theorem}
\begin{proof}
Simplifying \eqref{30} for the case $n=1$, we find that 
\begin{equation*}
{{\left| \left\langle x,y \right\rangle -\frac{\left\langle x,z
\right\rangle }{{{\left\| z \right\|}^{2}}}\left\langle z,y \right\rangle
\right|}^{2}}+\frac{{{\left| \left\langle x,z \right\rangle \right|}^{2}}}{{{%
\left\| z \right\|}^{2}}}{{\left\| y \right\|}^{2}}\le {{\left\| x \right\|}%
^{2}}{{\left\| y \right\|}^{2}}.
\end{equation*}
Apply these considerations to $x=Ax$, $y={{A}^{*}}x$ and $z=x$ with $\left\|
x \right\|=1$ we deduce 
\begin{equation}  \label{33}
{{\left( \frac{\left| \left\langle {{A}^{2}}x,x \right\rangle -{{%
\left\langle Ax,x \right\rangle }^{2}} \right|}{\left\| {{A}^{*}}x \right\|}
\right)}^{2}}+{{\left| \left\langle Ax,x \right\rangle \right|}^{2}}\le {{%
\left\| Ax \right\|}^{2}}.
\end{equation}
We denote the first expression on the left-hand side of \eqref{33} by $\xi
\left( x \right)$. Whence \eqref{33} implies that 
\begin{equation*}
\underset{\left\| x \right\|=1}{\mathop{\inf }}\,{{\xi }^{2}}\left( x
\right)+{{\left| \left\langle Ax,x \right\rangle \right|}^{2}}\le {{\left\|
Ax \right\|}^{2}}.
\end{equation*}
Now, the result follows by taking the supremum over all unit vectors in $%
\mathcal{H}$.
\end{proof}
\begin{remark}\label{3.3}
Of course, if $A$ is a normal operator we must have $\xi \left( x \right)=0$. In this regard, we have: 
\begin{enumerate}[(i)]
\item If $A$ is a normal matrix and $x$ is an eigenvector of $A$ with the eigenvalue $e$, then $\left\langle {{A}^{2}}x,x \right\rangle -{{\left\langle Ax,x \right\rangle }^{2}}={{e}^{2}}-{{e}^{2}}=0$.
\item Let $\sigma \left( A \right)$ and ${{\sigma }_{ap}}\left( A \right)$  be the spectrum and approximate spectrum of $A$, respectively. It is well-known that the spectrum of a normal operator has a simple structure. More precisely, if $A$ is normal, then we have $\sigma \left( A \right)={{\sigma }_{ap}}\left( A \right)$. If we assume that $e$ is in the approximate point spectrum of normal operator $A$, then there is a sequence ${{x}_{n}}\in \mathcal{H}$ with $\left\| {{x}_{n}} \right\|=1$ and $\left\langle A{{x}_{n}},{{x}_{n}} \right\rangle \to e$ as $n\to \infty $. Therefore $\underset{n\to \infty }{\mathop{\lim }}\,\left| \left\langle {{A}^{2}}{{x}_{n}},{{x}_{n}} \right\rangle -{{\left\langle A{{x}_{n}},{{x}_{n}} \right\rangle }^{2}} \right|=0$.
\end{enumerate}
\end{remark}
\section*{\bf Acknowledgments}
The authors would like to thank Professor Takeaki Yamazaki for his insightful comments  on Remark \ref{3.3}. Comments from the referee are also gratefully acknowledged.

\vskip 0.6 true cm

{\tiny $^1$Department of Mathematics, Mashhad Branch, Islamic Azad University, Mashhad, Iran.

{\it E-mail address:} hrmoradi@mshdiau.ac.ir

\vskip 0.4 true cm

{\tiny $^2$Department of Mathematics, Mashhad Branch, Islamic Azad University, Mashhad, Iran.

{\it E-mail address:} erfanian@mshdiau.ac.ir

\vskip 0.4 true cm

$^3$Department of Mathematics, Al-Balqa' Applied University, Salt 19117, Jordan.

{\it E-mail address:} khalid@bau.edu.jo}

\end{document}